\documentclass[a4paper, american, 11pt, reqno]{amsart}

\usepackage[latin1]{inputenc}
\usepackage{mathabx}
\usepackage{amsmath}
\usepackage{amssymb}
\usepackage{hyperref}
\usepackage{url}
\usepackage{babel}
\usepackage{pgf,tikz,pgfplots}
\pgfplotsset{compat=1.16}

\usepackage{graphicx}

\usepackage{tikz, float} \usetikzlibrary {positioning}
\usetikzlibrary{shapes.misc}
\usetikzlibrary{arrows}

\usetikzlibrary{calc}
\usepackage{amsfonts}
\input xy
\xyoption{all}

\newcommand{\Z}{{\mathbb Z}}
\newcommand{\N}{{\mathbb N}}

\newcommand{\R}{{\mathbb R}}

\newtheorem{thm}{Theorem}[section]

\newtheorem{lemma}[thm]{Lemma}

\newtheorem{remark}[thm]{Remark}

          {\theoremstyle{definition}
}
          {\theoremstyle{definition}

}

 \numberwithin{equation}{section}

\tikzset{%
  add/.style args={#1 and #2}{to path={%
 ($(\tikztostart)!-#1!(\tikztotarget)$)--($(\tikztotarget)!-#2!(\tikztostart)$)%
  \tikztonodes}}
} 

\newcommand{\comment}[1]{}

\begin{document}
\title[]{Convolution of a symmetric log-concave distribution and a symmetric bimodal distribution can have any number of modes}

\author[Charles Arnal]{Charles Arnal}

\address{Charles Arnal, Univ. Paris 6, IMJ-PRG, France.}
\email{charles.arnal@imj-prg.fr} 

\maketitle

\begin{abstract}
 In this note, we show that the convolution of a discrete symmetric log-concave distribution and a discrete symmetric bimodal distribution can have any strictly positive number of modes. A similar result is proved for smooth distributions, which contradicts the main statement in \cite{Faux}.

\end{abstract}

\tableofcontents

\newcommand\blfootnote[1]{%
  \begingroup
  \renewcommand\thefootnote{}\footnote{#1}%
  \addtocounter{footnote}{-1}%
  \endgroup
}

\section*{Acknowledgement}
The author is very grateful to Cl\'ement Deslandes for helpful discussions.\blfootnote{ This research was supported by the DIM Math Innov de la R\'egion Ile-de-France.
\begin{center}
\includegraphics[width = 20mm]{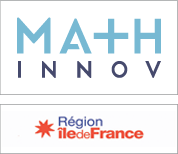} 
\end{center}  }

\section{Introduction}
Log-concave functions and sequences feature preeminently in many mathematical domains, including probability, statistics and combinatorics, but also more surprisingly algebraic geometry (read \cite{Survey1} and \cite{Survey2} for surveys of the notion).

Likewise, it is natural to consider the properties preserved by the convolution of two distributions, and in particular to consider the number of modes of the resulting distribution. It has been known for a long time that the convolution of two unimodal distributions need not be unimodal, though the convolution of two symmetric unimodal distributions will be unimodal (see \cite{DJDUnimodality}, or \cite{ProofAlternativeDiscrete} for an alternative proof of the second fact in the discrete case).

Further explorations have shown that convolutions of log-concave real distributions are log-concave (hence unimodal) (see \cite{ConvolutionLogConcave}), and that the convolution of a unimodal (but not symmetric) distribution with itself can have any number of modes (see \cite{ConvolutionAnyNumberModes}).
In the same line of questioning, we prove the following result:

\begin{thm}\label{TheoremDiscret}
Let $n\in\mathbb{N}$ be greater or equal to $1$. Then there exists a discrete log-concave distribution $p_n$  and a discrete bimodal distribution $q_n$, both symmetric about $0$, such that their convolution $p_n* q_n$ has exactly $n$ modes.
\end{thm}

We also prove a continuous variant, which directly contradicts the main statement\footnote{The mistake can be found in the proof of Theorem 1 in \cite{Faux}: Equations (3) and (4) do not imply Equation (5).} in \cite{Faux}:

\begin{thm}\label{TheoremContinu}
Let $n\in\mathbb{N}$ be greater or equal to $1$. Then there exists an absolutely continuous distribution on $\R$ whose density function $f_n$ is smooth, symmetric about $0$ and log-concave (hence unimodal), and an absolutely continuous distribution on $\R$ whose density function $g_n$ is smooth, symmetric about $0$ and bimodal, such that the convolution $f_n* g_n$ has at least $n$ modes.

\end{thm}

\begin{remark}
In fact, we can ask that $f_n* g_n$ have exactly $n$ modes; proving it makes the demonstration much more tedious, yet not any deeper.
\end{remark}

In what follows, the reader is reminded of the main definitions in Section \ref{SectionDefinitions}. Theorem \ref{TheoremDiscret} is proved in Section \ref{SectionDiscreteCase}, and Theorem \ref{TheoremContinu} is proved in Section \ref{SectionContinuousCase}.
\section{Definitions}\label{SectionDefinitions}
In this note, we will consider absolutely continuous distributions on $\R$ with continuous density functions and discrete distributions on $\Z$.

The \textit{modes} of an absolutely continuous distribution with continuous density function $f$  are the local maxima of $f$.

We will say that an absolutely continuous distribution with continuous density function $f$ is \textit{strictly $n$-modal in $\{a_1, \ldots, a_n \}$} if there exists $a_1<b_1<a_2< \ldots<b_{n-1}<a_n \in \R$ such that $f$ is strictly increasing on $[-\infty,a_1]$ and $[b_i,a_{i+1}]$ for $i=1,\ldots,n-1$, and strictly decreasing on $[a_{n},\infty]$ and $[a_i,b_i]$ for $i=1,\ldots,n-1$.  

We will say that a discrete distribution with probability mass function $p$ has a mode in $\{m, m +1,\ldots,m+k\}\subset \Z$ if $p(m-1)< p(m)=p(m+1)=\ldots=p(m+k)> p(m+k+1)$ .
We say that it is $n$-modal if it has exactly $n$ modes.

In the discrete case, a mode is often required to be a global maximum of the mass function. This nuance in definition matters not, as all our discrete modes will be global maxima.

The density function $f$ of an absolutely continuous distribution is called \textit{log-concave} if $f(\alpha x+(1-\alpha)y)\geq f(x)^\alpha f(y)^{1-\alpha} $ for all $x,y\in \R$ and all $\alpha \in [0,1]$. If $f$ is strictly positive, it is equivalent the concavity of $\log\circ f$.

The mass function $p$ of a discrete distribution is called \textit{log-concave} if $p(m)^2\geq p(m-1)p(m+1)$ for all $m\in \Z$, and if its support is a contiguous interval, \textit{i.e.} if there exists $m_1,m_2\in \Z$ such that $m_1<m_2$, $p(m) = 0 $ for all $m\leq m_1$ and all $m\geq m_2$, and $p(m)>0$ for all $m_1<m<m_2$. The second condition is sometimes omitted.

\begin{remark}
It is easy to show that log-concavity implies unimodality both in the continuous and the discrete case.
\end{remark}

As usual, the convolution of density functions $f,g:\R \longrightarrow \R$ (and by extension the convolution of the two associated absolutely continous distributions) is defined  as
$$ f* g (x) = \int _\R f(t)g(x-t)dt $$
for all $x\in \R$.

Similarly, the convolution of mass functions $p,q: \Z \longrightarrow [0,1] $ (and by extension the convolution of the two associated discrete distributions) is defined  as
$$ p* q (m) = \sum _{k\in \Z} p(k)q(m-k) $$
for all $m \in \Z$.

\section{Discrete case}\label{SectionDiscreteCase}

\begin{proof}[Proof of Theorem \ref{TheoremDiscret}]

The case $n=1$ is almost trivial: consider $p_1:=\frac{1_{\{-1,0,1 \}}}{3}$ and $q_1:=\frac{1_{\{-1,1 \}}}{2}$, where $1_A: \Z \longrightarrow \{0,1\}$ refers to the indicator function of $A$ for any $A\subset \Z$.

Now for $n\geq 2$, let us define $\tilde{p_n}=1_{\{-n+2,\ldots,0,\ldots,n-2 \}}$ and $p_n=\tilde{p_n}\cdot\frac{1}{2n-3}$. The function $p_n$ is clearly a log-concave distribution and symmetric about $0$.

If $n$ is even, let $\tilde{q_n}$ be as such:
for $k=1,\ldots,\frac{n}{2}-1$, we let $\tilde{q_n}(2k-1)=\tilde{q_n}(2k)=2k+1$ and $\tilde{q_n}(2n-2-2k)=\tilde{q_n}(2n-1-2k)=2k$.
We also let $\tilde{q_n}(n-1)=n$, $\tilde{q_n}(0)=1$ and $\tilde{q_n}(k)$ be $0$ for any $k>2n-3$.
We define $\tilde{q_n}$ on $\Z _{\leq 0}$ symmetrically.

Let $C_n:=\sum_{m\in \Z} \tilde{q_n}(m)$ and $q_n:=\frac{\tilde{q_n}}{C_n}$. Then $q_n$ is symmetric about $0$, and it is a bimodal distribution whose modes are in $\{-n+1,n-1\}$. It has a minimum in $0$.

See Figure \ref{DiscreteCase6Figure} to see the case $n=6$ illustrated.

\begin{figure}
\includegraphics[scale=0.65]{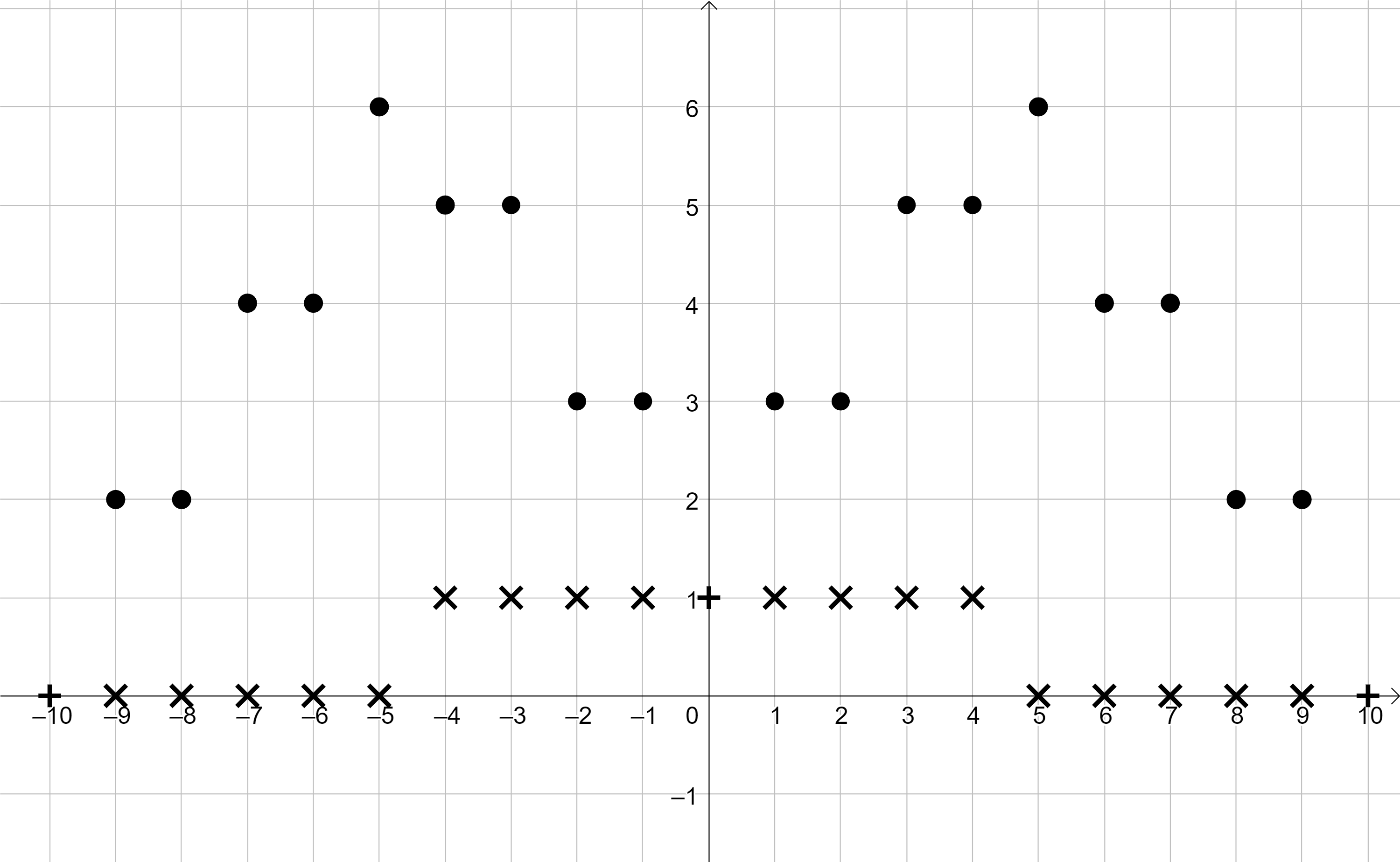}
\caption{
Case $n=6$. The $\bullet$ correspond to $\tilde{q_6}$, the $\times$ to $\tilde{p_6}$ and the $+$ to where they take the same value. }
\label{DiscreteCase6Figure}
\end{figure}

The convolution $p_n* q_n$ is clearly symmetric about $0$.

For any $m\in \Z$, the finite difference satisfies $D(p_n* q_n)(m)=(p_n* q_n)(m) -(p_n * q_n)(m-1)=(D(p_n)* q_n) (m)$. As $D(p_n)(l)$ is equal to $\frac{1}{2n-3} $ if $l=-n+2$, to $\frac{-1}{2n-3} $ if $l=n-1$ and to $0$ otherwise, we get that $D(p_n* q_n)(m)=\frac{q_n(m+n-2)-q_n(m-n+1)}{2n-3}$.

We see that $D(p_n* q_n)(m)$ is $0$ if $m>(2n-3)+(n-1)=3n-4$ and is strictly negative if $3n-4\geq m>n-1$. For $m=1, \ldots,n-1$, the finite difference $D(p_n* q_n)(m)$ is equal to $\frac{1}{(2n-3)C_n}$ if $m$ is odd and to $\frac{-1}{(2n-3)C_n}$ if $m$ is even.

By symmetry of $p_n$ and $q_n$, if  $m\leq 0$, we have that
\begin{gather*}
  D(p_n* q_n)(m)=\frac{q_n(m+n-2)-q_n(m-n+1)}{2n-3}=\\ -\frac{q_n(-m+n-1)-q_n(-m-n+2)}{2n-3} =\\ -\frac{q_n(-(m-1)+n-2)-q_n(-(m-1)-n+1)}{2n-3}=-D(p_n* q_n)(-(m-1)).  
\end{gather*}

From this, we get that $p_n* q_n$ has exactly $n$ modes, located in $-n+1, \ldots,-3,-1,1, 3,\ldots,n-1$.

\begin{figure}
\includegraphics[scale=0.55]{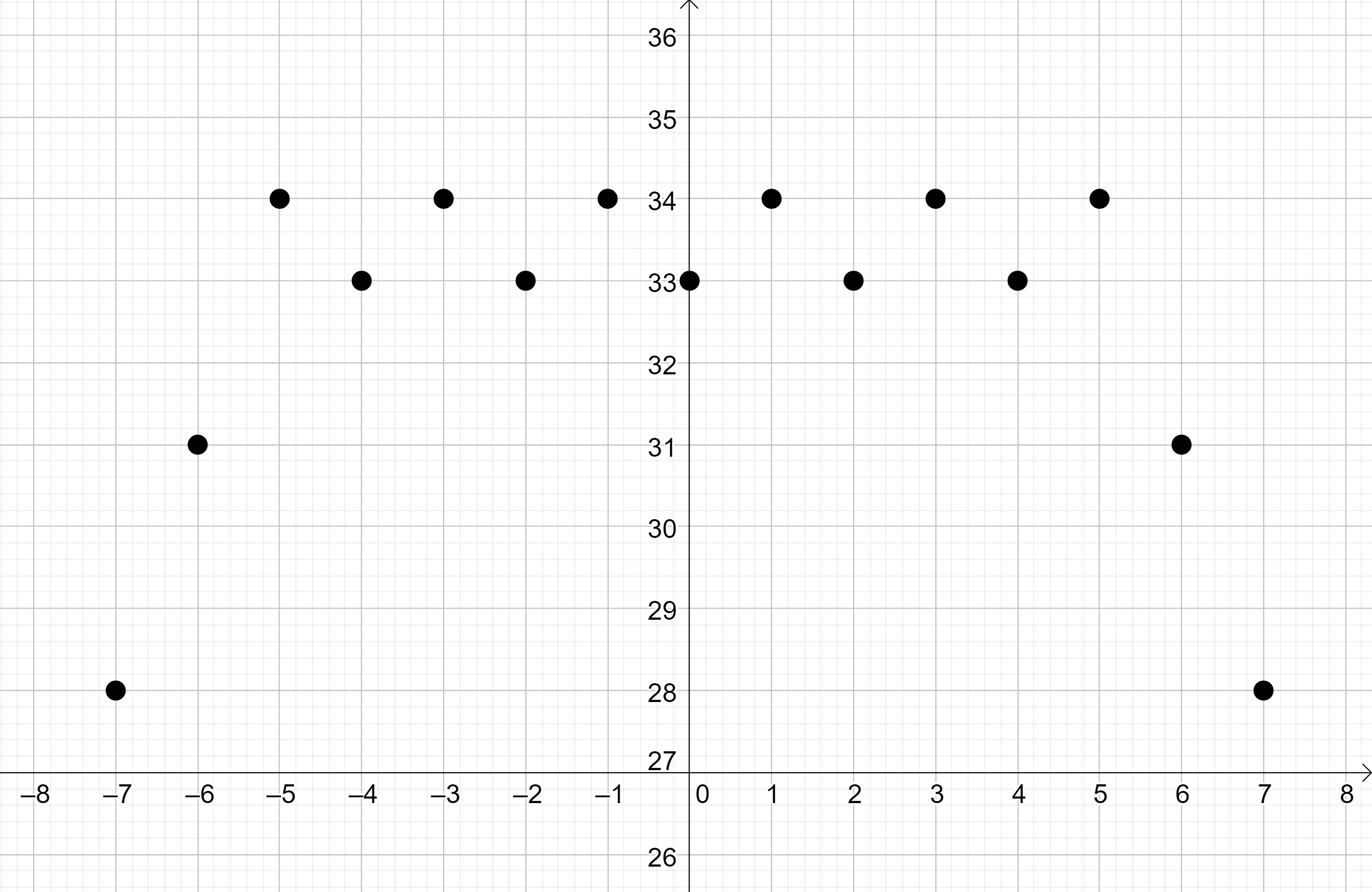}
\caption{
The convolution product $\tilde{p_n}*\tilde{q_n}$ in the case $n=6$.}
\label{ResultatN6Figure}
\end{figure}

The convolution $\tilde{p_n}* \tilde{q_n}$ is illustrated in Figure \ref{ResultatN6Figure} in the case $n=6$ (remember that $p_n* q_n = \frac{\tilde{p_n}* \tilde{q_n} }{(2n-3)C_n}$).

The case where $n$ is odd (and strictly greater than $1$) is very similar. Let $\tilde{q_n}$ be as such:
for $k=1,\ldots,\frac{n-1}{2}-1$ we let $\tilde{q_n}(2k-1)=\tilde{q_n}(2k)=2k+1$, and for $k=1,\ldots,\frac{n-1}{2}$ we let  $\tilde{q_n}(2n-2-2k)=\tilde{q_n}(2n-1-2k)=2k$.
We also let $\tilde{q_n}(n-2)=n$, $\tilde{q_n}(0)=1$ and $\tilde{q_n}(k)$ be $0$ for any $k>2n-3$.

As before, we define $\tilde{q_n}$ on $\Z _{\leq 0}$ symmetrically, let $C_n:=\sum_{m\in \Z} \tilde{q_n}(m)$ and $q_n:=\frac{\tilde{q_n}}{C_n}$. Then $q_n$ is symmetric about $0$, and it is a bimodal distribution whose modes are in $\{-n+2,n-2\}$. It has a minimum in $0$.

The case $n=7$ is illustrated in Figure \ref{DiscreteCase7Figure}.

As above, $D(p_n* q_n)(m)= \frac{q_n(m+n-2)-q_n(m-n+1)}{2n-3}$. Thus we see that $D(p_n* q_n)(m)$ is $0$ if $m>(2n-3)+(n-1)=3n-4$ and is strictly negative if $3n-4\geq m>n-1$. For $m=1, \ldots,n-1$, the finite difference $D(p_n* q_n)(m)$ is equal to $\frac{1}{(2n-3)C_n}$ if $m$ is even and to $\frac{-1}{(2n-3)C_n}$ if $m$ is odd.

By symmetry, $D(p_n* q_n)(m) = -D(p_n* q_n)(-(m-1))$ if $ m\leq 0$.

From this, we get that $p_n* q_n$ has exactly $n$ modes, located in $-n+1, \ldots,-2,0, 2,\ldots,n-1$.
The proof is complete.

\begin{figure}
\includegraphics[scale=0.55]{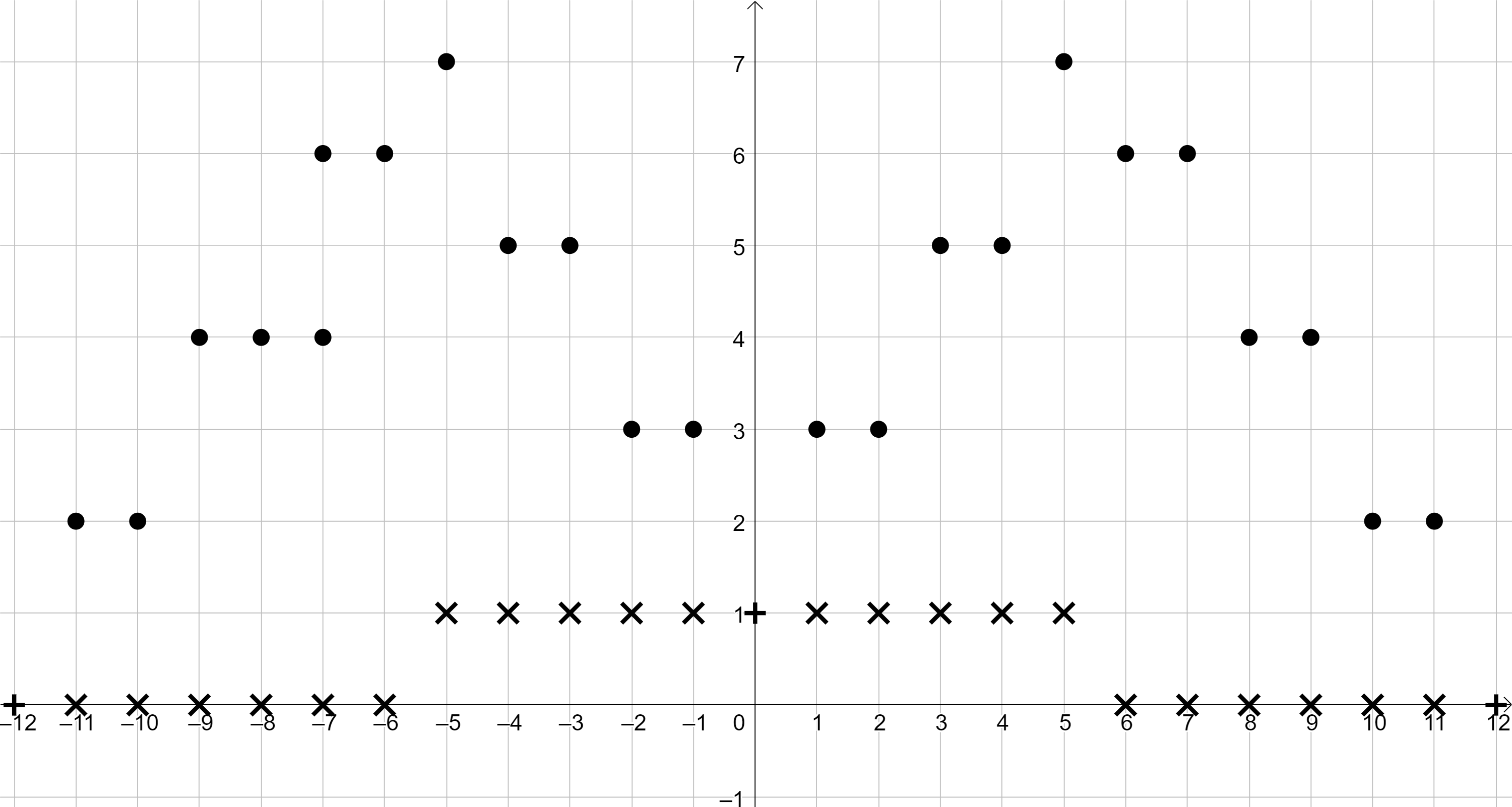}
\caption{
Case $n=7$. The $\bullet$ correspond to $\tilde{q_7}$, the $\times$ to $\tilde{p_7}$ and the $+$ to where they take the same value. }
\label{DiscreteCase7Figure}
\end{figure}

\end{proof}

\begin{remark}
Note that if we wanted $p_n$ to be such that there is a strict global maximum in $0$, we could replace $\tilde{p_n}$ in the proof of Theorem \ref{TheoremDiscret} by the restriction to $\Z$ of $x\mapsto exp\left(-\left[\frac{x}{n-2+\frac{1}{2}}\right]^{2i}\right)$ for $i\in \mathbb{N}$ large enough, and then proceed as above.
\end{remark}
\section{Continuous case}\label{SectionContinuousCase}

We come up with a smooth variant of the construction used in Section \ref{SectionDiscreteCase}.

Given a function $p:\Z \longrightarrow \Z$, we define $$\Phi (p) := \sum _{k\in \Z} p(k) 1_{]k-\frac{1}{2},k+\frac{1}{2}]}:\R \longrightarrow\R.$$ Note that $\Phi(p)|\Z =p$.

\begin{lemma}\label{TechnicalLemma1}
Let $p,q:\Z \longrightarrow \Z$ be two functions with finite support.

Then the convolution $\Phi(p) *_{\R} \Phi(q)$ over $\R$ (where $ \Phi(p)*_\R \Phi(q) (x) = \int _\R \Phi(p)(t)\Phi(q)(x-t)dt $)  is a continuous piecewise affine function which is affine on each interval $[l,l+1]$ for $l\in \Z$. Moreover, the convolution $p  *_{\Z} q$  over $\Z$ (where $  p*_\Z q (m) = \sum _{k\in \Z} p (k)q (m-k) $)  coincides with the restriction to $\Z$ of $\Phi(p) *_{\R} \Phi(q)$:
$$(\Phi(p) *_{\R} \Phi(q)) |\Z  = p  *_{\Z} q .$$
In particular, $p  *_{\Z} q$ and $\Phi(p) *_{\R} \Phi(q)$ have the same modes.
\end{lemma}
The proof of Lemma \ref{TechnicalLemma1} is straightforward.

Denote by $||- ||_\infty: f\mapsto ||f ||_\infty = sup_{x\in \R} (|f(x)|)$ the uniform norm. We need the following Lemma, which is easy to prove using classical smooth approximation tricks:

\begin{lemma}\label{TechnicalLemma2}
For any $n\in \N$, there exists a family of smooth symmetric about $0$ density functions $\{g_n^N\}_{N\in\N}$ such that:

\begin{enumerate}
    \item As $N$ goes to $\infty$, $g_n^N$ converges to $\Phi(q_n)$ in norm $L_1$.
    \item There exists $M>0$ such $||g_n^N||_\infty \leq M$ for all $N$.
    \item If $n$ is even (respectively, odd and strictly greater than $1$),  $g_n^N$ is strictly increasing from $-\infty$ to $-n+1$ (respectively $-n+2$) and strictly decreasing from $-n+1$ (respectively $-n+2$) to $0$ (and symmetrically so on $\R_+$) for all $N$.
    In particular, $g_n^N$ is bimodal.
    \item If $n=1$, $g_1^N$ is strictly increasing from $-\infty$ to $-1$  and strictly decreasing from $-1$ to $0$ (and symmetrically so on $\R_+$) for all $N$. In particular, $g_1^N$ is bimodal.

\end{enumerate}

\end{lemma}

\begin{proof}[Proof of Lemma \ref{TechnicalLemma2}]

We assume that $n\geq 2$ - the case $n=1$ is similar.
Given $N\in \N$, one can for example consider the function
$$h_N:=x\mapsto \frac{1}{1+\exp(\frac{Nx}{x^2-1})}1_{\{-1< x < 1\}} + 1_{\{x\geq 1\}},$$
illustrated in Figure \ref{FigureStepFunction}, which is smooth and approximates the Heavyside step function as $N\rightarrow \infty$.

\begin{figure}
\includegraphics[scale=3.5]{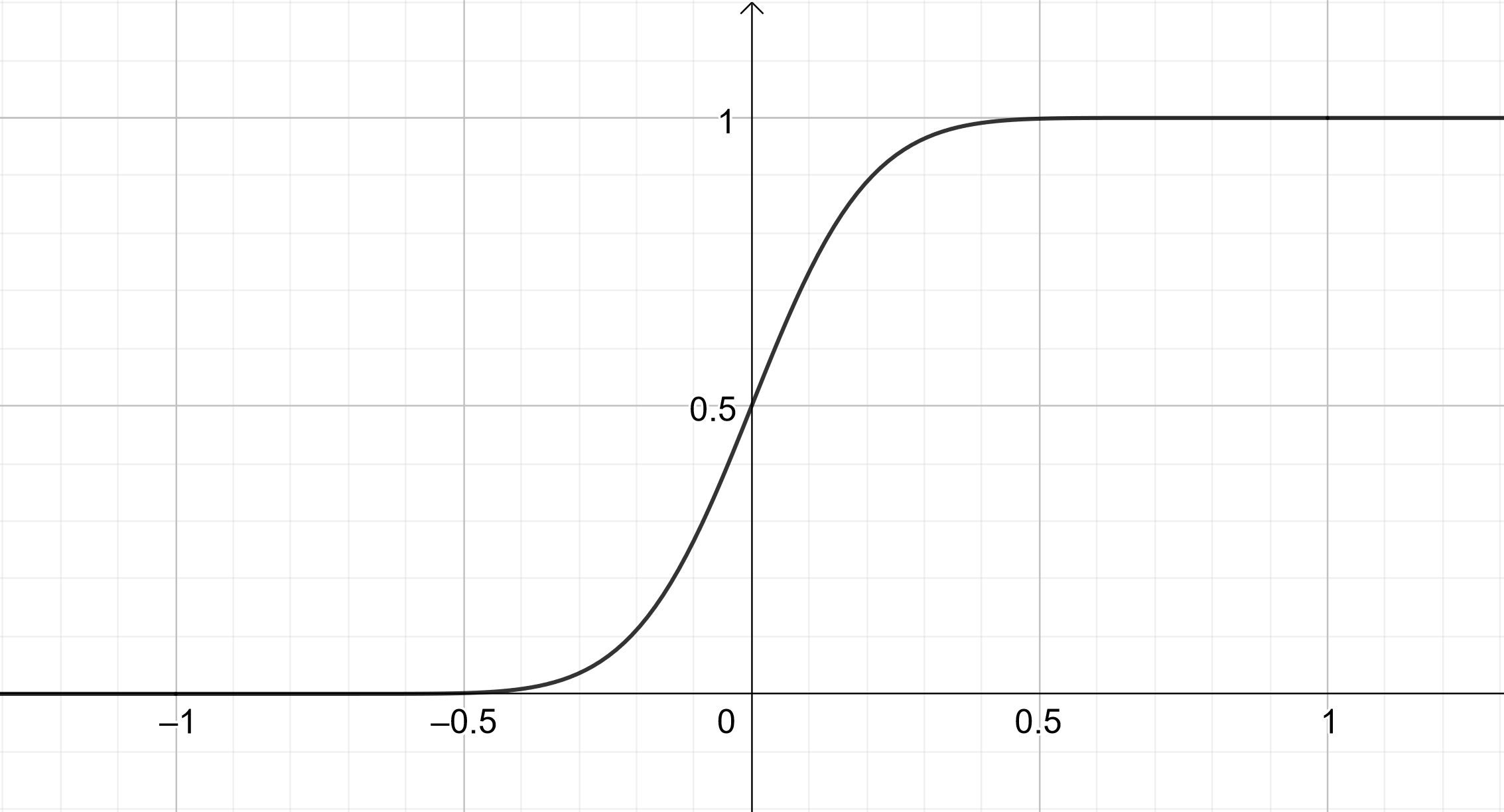}
\caption{The graph of $h_N:x \mapsto \frac{1}{1+\exp(\frac{Nx}{x^2-1})}1_{\{-1< x < 1\}} + 1_{\{x\geq1\}}$ for $N=10$.
}
\label{FigureStepFunction}
\end{figure}

Let also $b_N:\Z \rightarrow \R$ be such that $b_N(k)=0$ if $k\not \in \{-2n+3,\ldots,2n-3\}$, that $b_N(k)=\frac{1}{N}$ if $k\in \{-2n+3,\ldots,2n-3\}$ is even, and $b_N(k)=-\frac{1}{N}$ if $k\in \{-2n+3,\ldots,2n-3\}$  is odd.
Then for any $n\geq 2$  and $N\in \N$, one can define a smooth and symmetric about $0$ function $\tilde{g}_n^N$ as follows: let $\tilde{g}_n^N(x)=0$ for any $x<-2n+2$ or $x>2n-2$, and let
$$\tilde{g}_n^N(x)=\left(\tilde{q}_n(k+1)+b_N(k+1)-\tilde{q}_n(k)-b_N(k) \right)\cdot h_N(2(x-k-0.5))+\tilde{q}_n(k)+b_N(k)$$
for $x\in[k,k+1]$ and $k\in\{-2n+2,2n-2\}$, where $\tilde{q}_n$ is as in the proof of Theorem \ref{TheoremDiscret}. The case $n=5$ and $N=10$ is illustrated in Figure \ref{FigureSmoothN5}.
Then $\tilde{g}_n^N$ converges to $\Phi(\tilde{q}_n)$ in norm $L^1$ as $N\rightarrow \infty$, and the family of functions
$$g_n^N:= \frac{\tilde{g}_n^N}{\int_\R \tilde{g}_n^N(x) dx}$$
satisfies conditions (1) to (4).

\begin{figure}
\includegraphics[scale=0.75]{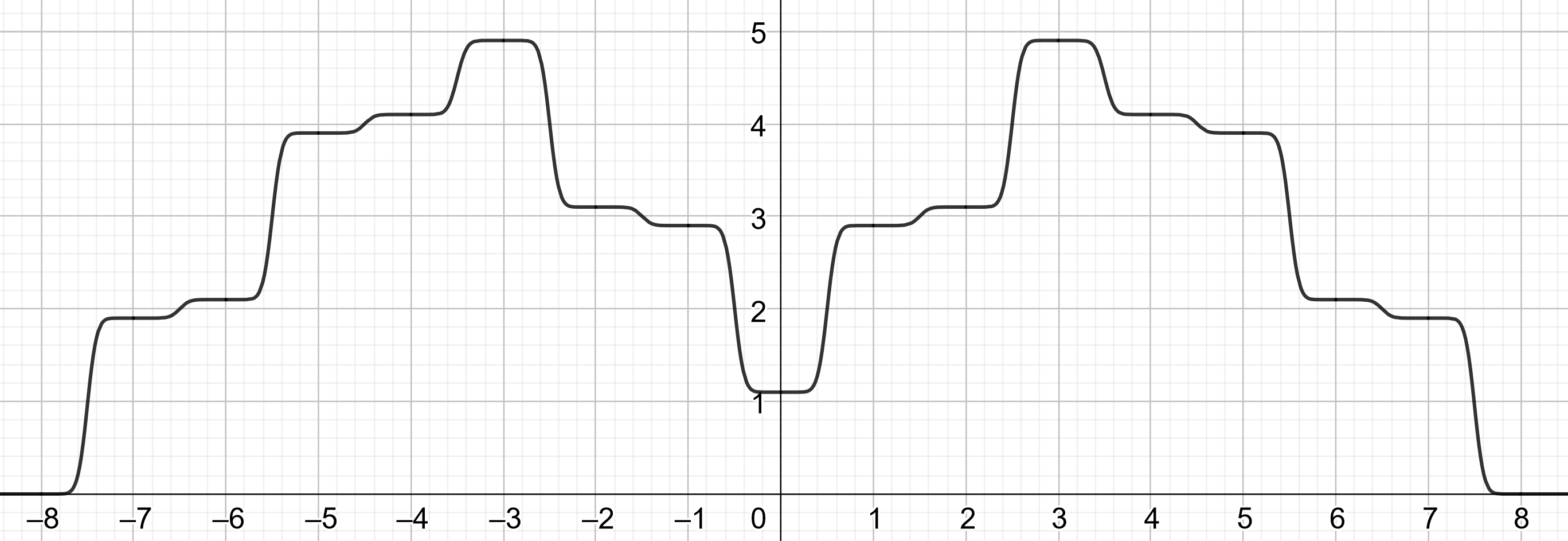}

\caption{The graph of $\tilde{g}_n^N$, which serves as a smooth approximation of $\Phi(\tilde{q}_n)$, for $n=5$ and $N=10$.
}
\label{FigureSmoothN5}
\end{figure}

\end{proof}
We can now prove Theorem \ref{TheoremContinu}.
\begin{proof}[Proof of Theorem \ref{TheoremContinu}]

Consider the mass functions $p_n$ and $q_n$ defined in the proof of Theorem \ref{TheoremDiscret}.

For any $N\in\N_{>0}$, let $f_n^N$ be defined by
$$f_n^N(x)=\frac{    exp(-(\frac{x}{  n-2+\frac{1}{2}  }  )^{2N})      } {\int_\R exp(-(\frac{y}{  n-2+\frac{1}{2}  }  )^{2N})  dy }$$
if $n\geq 2$ and 
$$ f_1^N(x)=\frac{    exp(-(\frac{2x}{3  }  )^{2N})      } {\int_\R exp(-(\frac{2y}{  3 }  )^{2N})  dy }$$
if $n=1$.

This defines a  family of smooth log-concave distributions that are symmetric about $0$.  Moreover, as $N$ goes to infinity, $f_n^N$ converges in norm $L^1$ to $\Phi(p_n)$.

Let $\{g_n^N\}_{N\in \N}$ be as in Lemma \ref{TechnicalLemma2}.
Using properties (1) and (2), we see that 
\begin{gather*}
     ||f_n^N * g_n^N - \Phi(p_n)*\Phi(q_n)||_\infty \leq \\ ||f_n^N * g_n^N - \Phi(p_n) * g_n^N||_\infty + ||\Phi(p_n)*g_n^N -\Phi(p_n)*\Phi(q_n)||_\infty \\  \leq M||f_n^N-\Phi(p_n)||_{L_1} + 2||g_n^N-\Phi(q_n)||_{L_1} \xrightarrow{N \rightarrow \infty} 0.
\end{gather*}
In particular, $(f_n^N * g_n^N)|\Z$ converges uniformly to $(\Phi(p_n)*\Phi(q_n))|\Z = p_n*q_n $ (see Lemma \ref{TechnicalLemma1}).

We have seen in the proof of Theorem \ref{TheoremDiscret} that if $n$ is even,
\begin{gather*}
     p_n*q_n (-n)<p_n*q_n (-n+1)>p_n*q_n (-n+2)<\ldots \\  <p_n*q_n (-1) >p_n*q_n (0)<p_n*q_n (1)>  \ldots \\ >p_n*q_n (n-2)<p_n*q_n (n-1)>p_n*q_n (n).
\end{gather*}

Hence for $N$ large enough, we also have 

\begin{align*}
  &  f_n^N * g_n^N (-n)<f_n^N * g_n^N (-n+1)>f_n^N * g_n^N (-n+2)<\ldots \\ & <f_n^N * g_n^N (-1) >f_n^N * g_n^N (0)<f_n^N * g_n^N (1)>  \ldots \\ & >f_n^N * g_n^N (n-2)<f_n^N * g_n^N (n-1)>f_n^N * g_n^N (n),
\end{align*}
which means (using Rolle's theorem) that $f_n^N * g_n^N$ has at least $n$ modes.
Moreover, the number of modes must be even, since $0$ is not a mode (and the convolution is symmetric with respect to $0$).

The same reasoning applies if $n$ is odd and strictly greater than $1$, in which case 
\begin{gather*}
     p_n*q_n (-n)<p_n*q_n (-n+1)>p_n*q_n (-n+2)<\ldots \\ 
     >p_n*q_n (-1) <p_n*q_n (0)>p_n*q_n (1)<  \ldots \\>p_n*q_n (n-2)<p_n*q_n (n-1)>p_n*q_n (n)
\end{gather*}

and

\begin{align*}
  &  f_n^N * g_n^N (-n)<f_n^N * g_n^N (-n+1)>f_n^N * g_n^N (-n+2)<\ldots \\ & >f_n^N * g_n^N (-1) <f_n^N * g_n^N (0)>f_n^N * g_n^N (1)<  \ldots \\ & >f_n^N * g_n^N (n-2)<f_n^N * g_n^N (n-1)>f_n^N * g_n^N (n)
\end{align*}
for $N$ large enough, which again means that  $f_n^N * g_n^N$ has at least $n$ modes.
Moreover, as $(g_n^N)'(-n+2-\frac{1}{2})=-(g_n^N)'(n-2+\frac{1}{2})>0 $ by property (3) we see that   the second derivative of $f_n^N * g_n^N$ in $0$ is negative if $N$ is large enough  by considering $(f_n^N)'$ (which converges to $\frac{1}{2}\left(-\delta_{-n+2-\frac{1}{2}}+ \delta_{n-2+\frac{1}{2}}\right) $, where $\delta_x$ is the Dirac distribution in $x\in \R$). Hence, $0$ is a mode of  $f_n^N * g_n^N $, which has an odd number of modes.

If $n=1$, likewise, $ p_1*q_1 (-1)<p_1*q_1 (0)>p_1*q_1 (1)$ and $f_1^N * g_1^N (-1)<f_1^N * g_1^N (0)>f_1^N * g_1^N (1)$ for $N$ large enough, which implies that  $f_1^N * g_1^N$ has at least one mode.
\end{proof}

\begin{remark}
By adding more technical conditions  and being more careful than we are in Lemma \ref{TechnicalLemma2} when defining the functions $g_n^N$, it can be tediously shown that we can make it so that $f_n^N*g_n^N$ has exactly $n$ modes.

The idea behind it is to make sure that the convolution alternates between being strictly convex and strictly concave, so that it admits exactly one mode on each interval on which it is convex. To do so, one needs to consider the second derivative $(f_n^N)'*(g_n^N)'$ of $f_n^N*g_n^N$, and use the fact that $(f_n^N)'$ converges nicely to a normalized sum of Dirac distributions and that $g_n^N$ can be chosen so that its derivative also converges to a  
weighted sum of Dirac distributions. Additional conditions concerning the behavior of $g_n^N$ on certain pairs of segments must also be added.

\end{remark}

\bibliographystyle{alpha}
\bibliography{Bibliography}

\end{document}